\documentclass[11pt,a4paper]{article}
\usepackage{amsmath, amsfonts,amsthm,amssymb,amsbsy,upref,color,graphicx,amscd,hyperref,makeidx,enumerate}
\usepackage[active]{srcltx}
\usepackage[latin1]{inputenc}

\def\L{\mathcal{L}}
\def\E{E}
\def\<{\langle}
\def\>{\rangle}

\def\B{\mathcal{B}}
\def\H{\mathcal{H}}

\def\Dr{D}

\def\D{\mathcal{D}}
\def\L{\mathcal{L}}
\def\g{\gamma}

\def\N{\mathbb{N}}
\def\O{\Omega}

\def\R{\mathbb{R}}
\def\A{\mathcal{A}}

\def\HH{\mathcal{H}}

\newcommand{\be}{\begin{equation}}
\newcommand{\ee}{\end{equation}}
\newcommand{\bib}[4]{\bibitem{#1}{\sc#2: }{\it#3. }{#4.}}

\newcommand{\cp}{\mathop{\rm cap}\nolimits}

\numberwithin{equation}{section}
\theoremstyle{plain}
\newtheorem{teo}{Theorem}[section]

\newtheorem{prop}[teo]{Proposition}
\newtheorem{deff}[teo]{Definition}
\theoremstyle{remark}
\newtheorem{oss}[teo]{Remark}
\newtheorem{exam}[teo]{Example}

\theoremstyle{figure}

\newenvironment{ack}{{\bf Acknowledgements.}}

\title{Some New Problems in Spectral Optimization}

\author{Giuseppe Buttazzo, Bozhidar Velichkov}

\begin{document}
\maketitle

\begin{abstract}
We present some new problems in spectral optimization. The first one consists in determining the best domain for the Dirichlet energy (or for the first eigenvalue) of the {\it metric Laplacian}, and we consider in particular Riemannian or Finsler manifolds, Carnot-Carath\'eodory spaces, Gaussian spaces. The second one deals with the optimal shape of a graph when the minimization cost is of spectral type. The third one is the optimization problem for a Schr\"odinger potential in suitable classes.
\end{abstract}



%%%%%%%%%%%%%%%%%%%%%%%%%%%%%%
\section{Introduction}\label{sintro}

Spectral optimization theory goes back to 1877, when Lord Raileigh conjectured, in his book {\it``The Theory of Sound''} \cite{ray77}, that among all drums of prescribed area the circular one had the lowest sound. Here are his precise words:

{\it``If the area of a membrane be given, there must evidently be some form of boundary for which the pitch (of the principal tone) is the gravest possible, and this form can be no other than the circle\dots''}

Since then, many other optimization problems involving the spectrum of the Laplace operator have been considered (see for instance the survey paper \cite{buremc} and the books \cite{book}, \cite{hen06}, \cite{pihe05}), showing the existence of optimal shapes and their qualitative properties together with the corresponding necessary conditions of optimality. However, in spite of the strong development of the theory, many problems still remain open and many conjectures are still waiting for a proof.

In this paper we present some different directions of research; our goal is to consider spectral optimization issues for the following three classes of problems.

\begin{itemize}

\item Optimization with respect to the domain for functionals like the Dirichlet energy or the first Dirichlet eigenvalue related to the {\it metric Laplacian}. This operator is in general non-linear and acts on functions defined on a general metric space; of particular interest are the cases when the metric space consists in a Riemannian or Finsler manifold, in a Carnot-Carath\'eodory space, in a Gaussian space.

\item Optimization of the shape of a {\it graph} with respect to the Dirichlet energy or to the first eigenvalue. In this case some explicit examples can be provided, together with some general necessary conditions of optimality.

\item Optimization of the potential $V(x)$ in a Schr\"odinger equation of the form $-\Delta u+V(x)u=f(x)$. The potential will be submitted to some suitable integral constraints and an existence result will be provided for several cost functionals.

\end{itemize}

The three cases above will be treated in Sections \ref{s2}, \ref{s3} and \ref{s4}, respectively. In all the cases Dirichlet boundary conditions will be considered; other kinds of boundary conditions would require completely different mathematical tools that in many cases are only partially developed. Our main concern is addressed to the existence of optimal solutions; other very interesting questions, like for instance the regularity of optimal solutions, have at present only limited and partial answers. In all the three cases, the existence of an optimal domain is obtained through the direct methods of the calculus of variations, that require two main ingredients: compactness of the space of competitors and semi-continuity of the cost functional. In the literature (see for instance \cite{book}) some useful topologies on the family of admissible domains have been introduced, in order to provide the necessary compactness properties. The semi-continuity of the cost functional is a more involved issue and requires some careful analysis.

The purpose of the present paper is not to provide new proofs or new results but mainly to illustrate the field of spectral optimization problems through some examples and to discuss some crucial issues by proposing some interesting problems that, to the best of our knowledge, are still open.

In Section \ref{s2} we consider the general framework of metric spaces, on which the {\it metric Laplacian} operator can be defined, together with the related energy and spectral eigenvalues. We recall a general existence result of an optimal domain, obtained in \cite{buve12}, and we show some related examples concerning Riemannian or Finsler manifolds, Carnot-Carath\'eodory spaces, Gaussian spaces.

In Section \ref{s3} we consider the case of spectral optimization problems for graphs, and in some cases we are able to provide explicitly the optimal shapes. We consider a natural convergence on the set of metric graphs in terms of the connectivity matrices of the graphs and the lengths of the edges. It is not hard to check that the spectral functionals we consider are continuous with respect to this convergence. On the other hand the family of admissible graphs endowed with such a convergence is not even complete, which gives raise to some counterexamples to the existence. Thus, we investigate the problem in a wider, more appropriate class of competitors.

In the last Section \ref{s4} we consider potentials for Schr\"odinger equations and the related optimization problems. In this case the admissible set of choices is just $L^1_+(\O)$, the set of positive integrable functions on $\O$, and the constraints are given by some integral inequalities. In this case, both the compactness of the optimizing sequences and the semi-continuity of the cost functional are quite involved questions, and the existence of optimal potentials is only known in some particular cases, leaving several interesting problems still open.

%%%%%%%%%%%%%%%%%%%%%%%%%%%%%%
\section{Spectral optimization in metric spaces}\label{s2}

In this section we consider spectral optimization problems in the class of subsets of some ambient metric space $(X,d)$ endowed with a finite Borel measure $m$. We do not assume any compactness or boundedness of $X$ with respect to the distance $d$. Our main assumption is the compactness of the inclusion $L^2(m)\subset H^1(X,m)$, where $H^1(X,m)$ is a Sobolev space of functions on $(X,m)$, which we define in each of the cases we consider.

\subsection{Metric measure spaces}

In \cite{buve12} we consider a separable metric space $(X,d)$ endowed with a finite Borel measure $m$ and a Riesz subspace $H$ of $L^2(m)$ satisfying the Stone property, i.e.
$$\hbox{if }u\in H,\quad\hbox{then }u\wedge1\in H\hbox{ and }|u|\in H.$$
Let $D:H\to L^2_+(m)$ be a convex, $1$-homogeneous map which is also local, i.e.
$$D(u\vee v)=Du\cdot I_{\{u>v\}}+Dv\cdot I_{\{u\le v\}},\quad\forall u,v\in H.$$
We consider $H$ endowed with the norm 
$$\|u\|_H=\left(\|u\|^2_{L^2}+\|Du\|^2_{L^2}\right)^{1/2}.$$
Moreover, we assume that
\begin{enumerate}[($\HH$1)]
\item the inclusion $i:H\hookrightarrow L^2$ is compact;
\item the norm of the gradient is lower semi-continuous with respect to the $L^2$ convergence, i.e. for each sequence $u_n$ bounded in $H$ and convergent in the strong $L^2$ norm to a function $u\in L^2(m)$, we have that $u\in H$ and
$$\int_X|Du|^2\,dm\le\liminf_{n\to\infty}\int_X|Du_n|^2\,dm;$$
\item the linear subspace $H\cap C(X)$, where $C(X)$ denotes the set of real continuous functions on $X$, is dense in $H$ with respect to the norm $\|\cdot\|_H$.
\end{enumerate}

An interesting example of subspace $H$ with the properties above is given by the Sobolev space $H^1(X,m)$ in the sense of Cheeger \cite{cheeger}.

For any set $\O\subset X$, we define the space
$$H_0(\O)=\big\{u\in H\ :\ \cp(\{u\ne0\}\setminus\O)=0\big\},$$
where the capacity $\cp(E)$ of a generic set $E\subset X$, is defined by
$$\cp(E)=\inf\big\{\|u\|^2_H\ :\ u\in H,\ u\ge0\hbox{ on }X,\ u\ge1\hbox{ in a neighbourhood of }E\big\}.$$

\begin{deff}\label{lb}
For each Borel set $\O$ and each $k\ge1$, we define 
\be\label{lbk}
\lambda_k(\O)=\inf_{K\subset H_0(\O)}\sup\Big\{\int_\O|Du|^2\,dm\ :\ u\in K,\ \int_\O u^2\,dm=1\Big\},
\ee
where the infimum is over all $k$-dimensional linear subspaces $K$ of $H_0(\O)$.
\end{deff}

\begin{deff}\label{en}
For each Borel set $\O$ and each $f\in L^2(\O,m)$, the Dirichlet energy of $\O$ is defined as
\be\label{eneq}
E_f(\O)=\inf\Big\{\frac{1}{2}\int_\O|Du|^2\,dm+\frac{1}{2}\int_\O u^2\,dm-\int_\O uf\,dm\ :\ u\in H_0(\O)\Big\}.
\ee
\end{deff}

\begin{oss}
In the cases when we have the inequality $\|u\|_{L^2(m)}\le C\|Du\|_{L^2(m)}$, for each $u\in H$, it is more convenient to define the energy $E_f(\O)$ as
\be\label{eneq1}
E_f(\O)=\inf\Big\{\frac{1}{2}\int_\O|Du|^2\,dm-\int_\O uf\,dm\ :\ u\in H_0(\O)\Big\}.
\ee
Also in this case the statement of the following theorem remains valid. 
\end{oss}

\begin{teo}\label{main}
Suppose that $(X,d)$ is a separable metric space with a finite Borel measure $m$ and suppose that $H\subset L^2(X,m)$ and $D:H\rightarrow L^2(X,m)$ are as above. Then the shape optimization problems 
$$\min\big\{E_f(\O)\ :\ \O\subset X,\ m(\O)\le1\big\},$$
and
$$\min\big\{\lambda_k(\O)\ :\ \O\subset X,\ m(\O)\le1\big\},$$
have solutions, which are quasi-open sets, i.e. level sets of the form $\{u>0\}$ for some function $u\in H$.
\end{teo}

\begin{oss}\label{general}
The existence result of Theorem \ref{main} holds, in the same form, for several other shape functionals $F(\O)$; the only required assumptions (see \cite{buve12}) are:

\begin{itemize}
\item[-] $F$ is monotone decreasing with respect to the inclusion, that is
$$F(\O_1)\le F(\O_2)\qquad\hbox{whenever }\O_2\subset\O_1;$$
\item[-] $F$ is $\gamma$-lower semi-continuous, that is
$$F(\O)\le\liminf_{n\to\infty}F(\O_n)\qquad\hbox{whenever }w_{\O_n}\to w_\O\hbox{ in }L^2(X,m)$$
where $w_\O$ is the solution of the minimization problem \eqref{eneq} with $f=1$.
\end{itemize}
For instance, the following cases belong to the class above.

\medskip{\it Integral functionals.} Given a right-hand side $f$ we consider the PDE formally written as
$$-\Delta u+u=f\hbox{ in }\O,\qquad u\in H_0(\O),$$
whose precise meaning is given through the minimization problem \eqref{eneq}, and which provides, for every admissible domain $\O$, a unique solution $u_\O$ that we assume extended by zero outside of $\O$. The cost $F(\O)=J(u_\O)$ is then obtained by taking
$$J(u)=\int_Xj\big(x,u(x)\big)\,dm$$
for a suitable integrand $j$. If $f\ge0$ and $j(x,\cdot)$ is decreasing, this cost verifies the conditions above.

\medskip{\it Spectral optimization.} For every admissible domain $\O$ we consider the eigenvalues $\lambda_k(\O)$ of Definition \ref{lb} and the spectrum $\lambda(\O)=\big(\lambda_k(\O)\big)_k$. Taking the cost
$$F(\O)=\Phi\big(\lambda(\O)\big)$$
we have that the assumptions above are satisfied as soon as the function $\Phi:[0,+\infty]^\N\to[0,+\infty]$ is lower semicontinuous and increasing, in the sense that
$$\displaylines{\lambda_k^h\to\lambda_k\quad
\forall k\in\N\ \ \Rightarrow\ \ \Phi(\lambda)
\le\displaystyle\liminf_{h\to\infty}\Phi(\lambda^h)\,,\cr
\hskip-1.25truecm\lambda_k\le\mu_k\quad
\forall k\in\N\ \ \Rightarrow\ \ \Phi(\lambda)\le\Phi(\mu)\,.\cr}$$
\end{oss}

\subsection{Finsler manifolds}

Consider a differentiable manifold $M$ of dimension $d$ endowed with a Finsler structure, i.e. with a map $F:TM\rightarrow[0,+\infty)$ which has the following properties:
\begin{enumerate}
\item $F$ is smooth on $TM\setminus\{0\}$;
\item $F$ is 1-homogeneous, i.e. $F(x,\lambda X)=|\lambda|F(x,X)$, $\forall \lambda\in\R$;
\item $F$ is strictly convex, i.e. the Hessian matrix $g_{ij}(x)=\frac{1}{2}\frac{\partial^2}{\partial X^i\partial X^j}[F^2](x,X)$ is positive definite for each $(x,X)\in TM$.
\end{enumerate}
With these properties, the function $F(x,\cdot):T_x M\rightarrow [0,+\infty)$ is a norm on the tangent space $T_xM$, for each $x\in M$. We define the gradient of a function $f\in C^\infty(M)$ as $Df(x):=F^\ast(x,df_x)$, where $df_x$ stays for the differential of $f$ at the point $x\in M$ and $F^\ast(x,\cdot):T^\ast_xM\to\R$ is the co-Finsler metric, defined for every $\xi\in T^\ast_xM$ as
$$F^\ast(x,\xi)=\sup_{y\in T_xM}\frac{\xi(y)}{F(x,y)}.$$

% Writing each tangent vector in the base $(\frac{\partial}{\partial x^1},\dots,\frac{\partial}{\partial x^d})$, induced by a local coordinate chart, we obtain an isomorphism between $\R^d$ and $T_x M$ and so, we can consider the dual norm $F^\ast$ with respect to the standard scalar product on $R^d$. We define the gradient of a function $f\in C^\infty(M)$ as $Df(x):=F^\ast(x,df_x)$, where $df_x$ stays for the differential of $f$ at the point $x\in M$. 
 The Finsler manifold $(M,F)$ is a metric space with the distance:
$$d_F(x,y)=\inf\Big\{\int_0^1 F(\gamma(t),\dot\gamma(t))\,dt\ :\ \gamma:[0,1]\to M,\ \gamma(0)=x,\ \gamma(1)=y\Big\}.$$
For any finite Borel measure $m$ on $M$, we define $H:=H^1_0(M,F,m)$ as the closure of the set of differentiable functions with compact support $C^\infty_c(M)$, with respect to the norm 
$$\|u\|:=\sqrt{\|u\|_{L^2(m)}^2+\|Du\|_{L^2(m)}^2}.$$
The functionals $\lambda_k$ and $E_f$ are defined as in \eqref{lb} and \eqref{en}, on the class of quasi-open sets, related to the $H^1(M,F,m)$ capacity. Various choices for the measure $m$ are available, according to the nature of the Finsler manifold $M$. For example, if $M$ is an open subset of $\R^d$, it is natural to consider the Lebesgue measure $m=\L^d$. In this case, the non-linear operator associated to the functional $\int F^\ast(x,du_x)^2\,dx$ is called Finsler Laplacian. On the other hand, for a generic manifold $M$ of dimension $d$, a canonical choice for $m$ is the Busemann-Hausdorff measure $m_F$, i.e. the $d$-dimensional Hausdorff measure with respect to the distance $d_F$. The non-linear operator associated to the functional $\int F^\ast(x,du_x)^2\,dm_F(x)$ is the generalization of the Laplace-Beltrami operator and its eigenvalues are defined as in \eqref{lbk}. In view of Theorem \ref{main}, we have the following existence results:

\begin{teo}\label{fint}
Given a compact Finsler manifold $(M,F)$ with Busemann-Hausdorff measure $m_F$, the following problems have solutions:
$$\min\Big\{\lambda_k(\O)\ :\ m_F(\O)\le c,\ \O\hbox{ quasi-open, }\O\subset M\Big\},$$
$$\min\Big\{E_f(\O)\ :\ m_F(\O)\le c,\ \O\hbox{ quasi-open, }\O\subset M\Big\},$$
for any $k\in\N$, $0<c\le m_F(M)$ and $f\in L^2(M,m_F)$.
\end{teo}

\begin{teo}\label{finlap}
Consider an open set $M\subset\R^d$ endowed with a Finsler structure $F$ and the Lebesgue measure $\L^d$. If the diameter of $M$ with respect to the Finsler metric $d_F$ is finite, then the following problems have solutions:
$$\min\Big\{\lambda_k(\O)\ :\ |\O|\le c,\ \O\hbox{ quasi-open, }\O\subset M\Big\},$$
$$\min\Big\{E_f(\O)\ :\ |\O|\le c,\ \O\hbox{ quasi-open, }\O\subset M\Big\},$$
where $k\in\N$, $|\O|$ denotes the Lebesgue measure of $\O$, $c$ is a constant such that $0<c\le|M|$ and $f\in L^2(M)$.
\end{teo}

\begin{oss}
In \cite{kawohl} it was shown that if the Finsler metrics $F(x,\cdot)$ on $\R^d$ does not depend on $x\in\R^d$, then the solution of the optimization problem 
$$\min\Big\{\lambda_1(\O)\ :\ |\O|\le c,\ \O\hbox{ quasi-open, }\O\subset\R^d\Big\},$$
is the ball of measure $c$. It is clear that it is also the case when in the hypotheses of Theorem \ref{finlap} one considers $c>0$ such that there is a ball of measure $c$ contained in $M$. On the other hand , if $c$ is big enough the solution is not, in general, the geodesic ball in $M$ (see \cite{heou01}). If the Finsler metric is not constant in $x$, the solution will not be a ball even for small $c$. In this case it is natural to ask whether the optimal set gets close to the geodesic ball as $c\to0$. In \cite{sicbaldi} this problem was discussed in the case when $M$ is a Riemannian manifold. The same question for a generic Finsler manifold is still open.
\end{oss}

%%%%%%%%%%%%%%%%%%%%%%%%%%%%%%%%%%%%%%%%%%%%%%%%%%
\subsection{Gaussian spaces}\label{sgauss}

Consider the Euclidean space $\R^2$ endowed with the Gaussian measure 
$$m=(2\pi)^{-1}\,\exp\left(-\frac{x_1^2+x_2^2}{2}\right)dx_1dx_2.$$
Note that an orthonormal basis on $L^2(m)$ is given by the functions $H_{n,k}(x_1,x_2):=H_n(x_1)H_k(x_2)$, $n,k\in\N$, where $H_n:\R\to\R$ are the Hermite polynomials 
$$H_{n}(x):=\frac{(-1)^n}{\sqrt{n!}}\,\exp(x^2/2)\,\partial_x^n\left(\exp(-x^2/2)\right),$$
which satisfy 
$$\partial_x H_n(x)=\sqrt{n}H_{n-1}(x),\qquad \partial^2_xH_n(x)-xH_n(x)=nH_n(x).$$
We define the Sobolev space $W^{1,2}(\R^2,m)$ as
\be
W^{1,2}(\R^2,m)=\left\{u\in L^2(m)\ :\ |\nabla u|\in L^2(m)\right\},
\ee
where $\nabla u$ is the distributional gradient of $u$. It can be characterized using the basis $\left\{H_{n,k}\right\}_{n,k}$ as
\be
W^{1,2}(\R^2,m)=\big\{u\in L^2(m)\ :\ \sum_{n,k}(n+k)u_{n,k}^2<+\infty\big\}, 
\ee
where $u_{n,k}:=\int_{\R^2}H_{n,k}u\,dm$. At this point it is clear that the inclusion $W^{1,2}(\R^2,m)\subset L^2(m)$ is compact and that the estimate $\|u\|_{L^2(m)}\le\|\nabla u\|_{L^2(m)}$ holds. Moreover, the linear combinations of Hermite polynomials are dense in $W^{1,2}(\R^2,m)$ and so $C^\infty(\R^2)\cap W^{1,2}(\R^2,m)$ is dense in $W^{1,2}(\R^2,m)$. Thus, we can define the capacity $\cp(E)$ of any set $E\subset\R^2$ and the space $W^{1,2}_0(\O,m)$ of functions $u\in W^{1,2}(\R^2,m)$ such that $\cp(\{u\ne0\}\cap\O^c)=0$. For any $f\in L^2(m)$, there is a unique $w\in W^{1,2}_0(\O,m)$, which minimizes the functional 
$$J_f(u)=\frac12\int_\O|\nabla u|^2\,dm-\int_\O fu\,dm,$$
and defines the energy of $\O$ as $E_f(\O):=J_f(w)$.
We note that for any $v\in W^{1,2}_0(\O,m)$ we have
$$\int_{\O}\nabla w\cdot\nabla v\,dm=\int_{\O}fw\,dm,$$
and so, we say that $w$ is the weak solution of the problem $-\Delta w+x\cdot\nabla w=f$ in $W^{1,2}_0(\O,m)$. Since $\|\nabla u\|_{L^2(m)}\le\|f\|_{L^2(m)}$, we have that the operator $R_\O:L^2(m)\to L^2(m)$, which associates to each $f\in L^2(m)$ the function $R_\O(f):=w$, is compact. Thus $R_\O$ is the resolvent of an operator $-\Delta+x\cdot\nabla$, which is the Ornstein-Uhlenbeck operator on $\O$ and which has a discrete spectrum $\sigma(\O)$, given by the sequence $0\le\lambda_1(\O)\le\lambda_2(\O)\le\dots$. Note that, in the case $\O=\R^2$, the spectrum is given by $\sigma(\R^2)=\{n+k:\ n,k\in\N\}$. In particular, $\lambda_1(\R^2)=0$ and $\lambda_2(\R^2)=\lambda_3(\R^2)=1$. We also note that the $k$-th eigenvalue $\lambda_k(\O)$ can be represented as in \eqref{lb} and so, if $\O\ne\R^2$, then $\lambda_1(\O)>0$. Applying Theorem \ref{main}, we obtain the existence of optimal domains for any $\lambda_k$.

\begin{teo}\label{maingauss}
Consider $\R^2$ endowed with a non-degenerate Gaussian measure $m$, i.e. with invertible covariance matrix. Then, for any $k\in\N$, $f\in L^2(m)$ and $0\le c\le1$, the following optimization problems have solutions:
$$\min\Big\{\lambda_k(\O)\ :\ \O\subset\R^2,\ m(\O)\le c\Big\},$$
$$\min\Big\{E_f(\O)\ :\ \O\subset\R^2,\ m(\O)\le c\Big\},$$
which are quasi-open sets.
\end{teo}

\begin{oss}\label{gausspen}
Theorem \ref{maingauss} also applies to penalized problems, i.e. for any $\Lambda>0$, $k\in\N$ and $f\in L^2(m)$, there is a solution of the problems 
\be\label{gausspen1}
\min\Big\{\lambda_k(\O)+\Lambda m(\O)\ :\ \O\subset\R^2\Big\},
\end{equation}
\begin{equation}\label{gausspen2}
\min\Big\{E_f(\O)+\Lambda m(\O)\ :\ \O\subset\R^2\Big\},
\ee
which is a quasi-open set. As we will see in the example below, these problems are sometimes easier to threat when comes to regularity questions and qualitative study of the optimal sets. 
\end{oss}

\begin{exam}\label{exhalf}
Let $f$ be the constant $1$ in $\R^d$. By Remark \ref{gausspen}, the problem \eqref{gausspen2} has a solution $\O$, which we assume to be open and with boundary $\partial\O$ of class $C^2$ (that we expect to be true), we can perform the shape derivative of the energy $E_1$ with respect to some vector field $V$ regular enough. Indeed, following \cite[Chapter 5]{pihe05}, let $V:\R^d\to\R^d$ be a $C^\infty_c$ vector field and for each $t>0$ small enough, define $\Phi_t(x)=x+tV(x)$ and $\O_t=\Phi_t(\O)$. Then, we have
\be\label{dwdt}
\frac{dE_1(\O_t)}{dt}\Big|_{t=0}=-\frac12\int_\O w'\,dm,
\ee
where $w'$ is the solution of 
\be
\begin{cases}
-\Delta w'+x\cdot\nabla w'=0,\hbox{ in }\O,\\
w'=-V\cdot\nabla w,\hbox{ on }\partial\O.
\end{cases}
\ee
We denote with $w$ the (strong) solution of  
$$-\Delta w+x\cdot\nabla w=1,\qquad w\in W^{1,2}_0(\O,m),$$ 
and integrate by parts in \eqref{dwdt} obtaining 
\be\label{dEdt}
\frac{dE_1(\O_t)}{dt}\Big|_{t=0}=-\frac12\int_\O (-\Delta w+x\cdot\nabla w) w'\,dm=-\frac{1}{4\pi}\int_{\partial\O}\left|\frac{\partial w}{\partial n}\right|^2 V\cdot n\, e^{-|x|^2/2}\,d\H^{d-1},
\ee
where $n$ is the exterior normal on $\partial\O$ and $w$ is the energy function on $\O$, that is the solution of the Ornstein-Uhlenbeck PDE 
$$-\Delta w+x\cdot\nabla w=1\quad\hbox{in }\O,\qquad w\in W^{1,2}_0(\O,m).$$
On the other hand, we have 
\be\label{dmdt}
\frac{dm(\O_t)}{dt}\Big|_{t=0}=\frac{1}{2\pi}\int_{\partial\O}e^{-|x|^2/2}\,V\cdot n\,d\H^{d-1},
\ee
and so, by the optimality of $\O$, 
$$\left(\frac{dE_1(\O_t)}{dt}+\Lambda\frac{dm(\O_t)}{dt}\right)\Big|_{t=0}=0$$
for any vector field $V$. By \eqref{dEdt} and \eqref{dmdt} we obtain
$$\left|\frac{\partial w}{\partial n}\right|=\sqrt{2\Lambda}\qquad\hbox{on }\partial\O.$$

Summarizing, we have obtained that if an optimal domain $\O$ is regular enough, then the following overdetermined boundary value problem has a solution:
\be\label{over}
\begin{cases}
\begin{array}{ll}
-\Delta w+x\cdot\nabla w=1,&\hbox{ in }\O,\\
w=0,&\hbox{ on }\partial\O,\\
\frac{\partial w}{\partial n}=-\sqrt{2\Lambda},&\hbox{ on }\partial\O.
\end{array}
\end{cases}
\ee
It is straightforward to check that the following domains satisfy this condition:
\begin{itemize}
\item the half-space $\O=\{x_1>c\}$, for a given $c\in\R$,
\item the strip $\O=\{|x_1|<a\}$, for some $a>0$,
\item the euclidean ball $\O=\{|x|<r\}$, for some $r>0$,
\item the external domain of a ball $\O=\{|x|>r\}$, for $r>0$.
\end{itemize}
We do not know which of these domains is optimal and if there are other domains $\O$ for which the overdetermined problem \eqref{over} has a solution.
\end{exam}

%%%%%%%%%%%%%%%%%%%%%%%%%%%%%%%%%%%%%%%%%%%%%%%%%%
\subsection{Carnot-Carath\'eodory spaces}

Consider a bounded open and connected set $\Dr\subset \R^d$ and $C^{\infty}$ vector fields $Y_1,\dots,Y_n$ defined on a neighbourhood $U$ of $\overline\Dr$. We say that the vector fields satisfy the H\"ormander's condition on $U$, if the Lie algebra generated by $Y_1,\dots,Y_n$ has dimension $d$ in each point $x\in U$. 

 We define the Sobolev space $W^{1,2}_0(\Dr;Y)$ on $\Dr$ with respect to the family of vector fields $Y=(Y_1,\dots,Y_n)$ as the closure of $C^\infty_c(\Dr)$ with respect to the norm  
$$\|u\|_{Y}=\left(\|u\|_{L^2}^2+\sum_{j=1}^n \|Y_j u\|_{L^2}^2\right)^{1/2},$$
where the derivation $Y_ju$ is intended in sense of distributions.
For $u\in W^{1,2}_0(\Dr;Y)$, we define the gradient $Yu=(Y_1u,\dots,Y_nu)$ and set $|Yu|=\left(|Y_1u|^2+\dots+|Y_nu|^{2}\right)^{1/2}\in L^2(\Dr)$.

Setting $Du:=|Yu|$ and $H:=W^{1,2}_0(\Dr;Y)$, we define, for any $\O\subset\Dr$, the energy $E_f(\O)$ and the $k^{th}$ eigenvalue $\lambda_k(\O)$ of the operator $Y_1^2+\dots+Y_n^2$, as in \eqref{en} and \eqref{lb}. The following existence result is a consequence of Theorem \ref{main}.

\begin{teo}\label{cc}
Consider a bounded open set $\Dr\subset\R^d$ and a family $Y=(Y_1,\dots,Y_n)$ of $C^{\infty}$ vector fields defined on an open neighbourhood $U$ of the closure $\overline\Dr$ of $\Dr$. If $Y_1,\dots,Y_n$ satisfy the H\"ormander condition on $U$, then for any $k\in\N$, $0<c\le|\Dr|$ and $f\in L^2(\Dr)$, the following shape optimization problems admit a solution:
\be\label{cck}
\min\Big\{\lambda_k(\O)\ :\ \O\subset\Dr,\ \O\hbox{ quasi-open, }|\O|\le c\big\},
\ee
\be\label{cce}
\min\Big\{E_f(\O):\ \O\subset\Dr,\ \O\hbox{ quasi-open, }|\O|\le c\big\}.
\ee
\end{teo}

\begin{proof}
It is straightforward to check that the space $H:=W^{1,2}_0(\Dr;Y)$ and the application $Du:=|Yu|$ satisfy the assumptions of Theorem \ref{main}. The only non-trivial claim is the compact inclusion $H\subset L^2(\Dr)$, which follows since $Y_1,\dots,Y_n$ satisfy the H\"ormander condition on $U$. In fact, by the H\"ormander Theorem (see \cite{hormander}), there is some $\epsilon>0$ and some constant $C>0$ such that for any $\varphi\in C^\infty_c(\Dr)$
\be
\|\varphi\|_{H^{\varepsilon}}\le C\left(\|\varphi\|_{L^2}+\sum_{j=1}^{k}\|Y_j\varphi\|_{L^2}\right),
\ee
where we set
$$\|\varphi\|_{H^\varepsilon}=\left(\int_{\R^d}|\widehat\varphi(\xi)|^2(1+|\xi|^2)^\varepsilon\,d\xi\right)^{1/2},$$
being $\widehat\varphi$ the Fourier transform of $\varphi$. Let $H^\varepsilon_0(\Dr)$ be the closure of $C^\infty_c(\Dr)$ with respect to the norm $\|\cdot\|_{H^\varepsilon}$. Since the inclusion $L^2(\Dr)\subset H^\varepsilon_0(\Dr)$ is compact, we have the conclusion.
\end{proof}

\begin{oss}
In the hypotheses of Theorem \ref{cc}, the following optimization problems have a solution:
\be\label{cck1}
\min\Big\{\lambda_k(\O)+\Lambda|\O|\ :\ \O\subset\Dr,\ \O\hbox{ quasi-open}\big\},
\ee
\be\label{cce1}
\min\Big\{E_f(\O)+\Lambda|\O|:\ \O\subset\Dr,\ \O\hbox{ quasi-open}\big\},
\ee
where $k\in\N$, $\Lambda>0$ and $f\in L^2(\Dr)$ are given.
\end{oss}

\begin{exam}
Consider a bounded open set $\Dr\subset\R^2$ and the vector fields $X=\frac{\partial}{\partial x}$ and $Y=x\frac{\partial}{\partial y}$. Since $[X,Y]=\frac{\partial}{\partial y}$, we can apply Theorem \ref{cc} and so, the shape optimization problem \eqref{cce1} has a solution $\O\subset\Dr$. Assuming that $\O$ is regular enough we may repeat the argument from Section \ref{sgauss}. Indeed, suppose that $V$ is a vector field on $\partial\O$ and note that the map $\Phi_t=Id+tV$ is a differomorphism for $t$ small enough. Defining $\O_t=\Phi_t(\O)$ and $w$ the (strong) solution of 
\be\label{ccw}
-\left(\partial_x^2+x^2\partial_y^2\right)w+w=f,\qquad w\in W^{1,2}_0(\O;X,Y),
\ee
where $f\in L^2(\Dr)$, we have that 
\be\label{ccdedt}
\frac{dE_f(\O_t)}{dt}\Big|_{t=0}=-\frac{1}{2}\int_{\O}fw'\,dx,
\ee
where $w'$ is the weak solution of 
$$-\left(\partial_x^2+x^2\partial_y^2\right)w'+w'=0,\qquad w'+V\cdot\nabla w\in W^{1,2}_0(\O_t;X,Y).$$
Using \eqref{ccw} and integrating by parts in \eqref{ccdedt}, we obtain
\be\label{ccdedt1}
\frac{dE_f(\O_t)}{dt}\Big|_{t=0}=-\frac{1}{2}\int_{\partial\O}(V\cdot\nabla w)\left(n\cdot(\partial_xw,x^2\partial_y w)\right)\,d\H^{1}.
\ee
Since
\be\label{ccdmdt}
\frac{d|\O_t|}{dt}\Big|_{t=0}=\int_{\partial\O}V\cdot n\,d\H^{1},
\ee
we have that the energy function $w$ is a solution of the following overdetermined boundary value problem on the optimal set $\O$
\be\label{eqcc}
\begin{cases}
\begin{array}{ll}
-\left(\partial_x^2+x^2\partial_y^2\right)w+w=f&\hbox{ in }\O,\\
w=0&\hbox{ on }\partial\O,\\
\left(n\cdot(\partial_xw,x^2\partial_y w)\right)\frac{\partial w}{\partial n}=2\Lambda&\hbox{ on }\partial\O.
\end{array}
\end{cases}
\ee
The characterization of the solutions of \eqref{eqcc} is an open problem even in the case $f=1$.
\end{exam}

%%%%%%%%%%%%%%%%%%%%%%%%%%%%%%
\section{Spectral optimization for metric graphs}\label{s3}
In this section we study the problem of the optimization of the torsion rigidity of a one dimensional structure in $\R^d$ connecting a prescribed set of fixed points. Before we introduce the optimization problem we will examine some of the basic tools from the analysis of one dimensional sets.

Consider a closed connected set $C\subset\R^d$ of finite length $\H^1(C)<\infty$, where by $\H^1$ we denote the one-dimensional Hausdorff measure in $\R^d$. The natural choice of a distance on $C$ is 
$$d_C(x,y)=\inf\left\{\int_0^1|\dot\gamma(t)|\,dt\ :\ \gamma:[0,1]\to\R^d\hbox{ Lipschitz, }\g([0,1])\subset C,\ \g(0)=x,\ \g(1)=y\right\},$$
which, in turn, gives a pointwise definition of a gradient 
$$|u'|(x)=\limsup_{y\to x}\frac{|u(y)-u(x)|}{d(x,y)},$$
which is a function in $L^2(\H^1)$, at least in the case when $u:C\to\R$ is Lipschitz with respect to the distance $d_C$. For any function $u:C\to\R$, Lipschitz with respect to the distance $d_C$, we define the norm
$$\|u\|^2_{H^1(C)}=\int_C u^2\,d\H^1+\int_C |u'|^2\,d\H^1,$$
and the Sobolev space $H^1(C)$, as the closure of the Lipschitz functions on $C$ with respect to this norm.  
By the Second Rectifiability Theorem (see \cite[Theorem 4.4.8]{at}) the set $C$ consists of a countable family of injective arc-length parametrized Lipschitz curves $\g_i:[0,l_i]\to C$, $i\in\N$, i.e. there is an  $\H^1$-negligible set $N\subset C$ such that $C=N\cup\left(\cup_i\,\gamma_i([0,l_i])\right)$. On each curve $\gamma_i$ we have the chain rule $\Big|\frac{d}{dt}u(\g_i(t))\Big|=|u'|(\g_i(t))$ (see \cite[Lemma 3.1]{buruve} for a proof) and thus, we obtain the following expression for the norm of $u\in H^1(C)$:
\be\label{norm}
\|u\|^2_{H^1(C)}=\int_C u^2\,d\H^1+\sum_i\int_0^{l_i}\left|\frac{d}{dt}u(\gamma_i(t))\right|^2\,dt.
\ee

Given a set of distinct points $D_1,\dots,D_k\in\R^d$ we define the admissible class $\A_C(D_1,\dots,D_k)$ as the family of closed connected sets $C\subset\R^d$ containing $D_1,\dots,D_k$. For any $C\in \A_C(D_1,\dots,D_k)$ we consider the space of Sobolev functions which satisfy a Dirichlet condition at the points $D_i$:
$$H^1_0(C;D_1,\dots,D_k)=\{u\in H^1(C)\ :\ u(D_j)=0,\ j=1\dots,k\}.$$
For the points $D_j$ we use the term \emph{Dirichlet points}. The \emph{Dirichlet Energy} of the set $C$ with respect to $D_1,\dots,D_k$ is defined as
\be\label{energy1}
\E(C;D_1,\dots,D_k)=\min_{u\in H^1_0(C;D_1,\dots,D_k)}\frac12\int_C|u'|^2\,d\H^1-\int_C u\,d\H^1.
\ee

We study the following shape optimization problem:
\be\label{minenergy1}
\min\left\{\E(C;D_1,\dots,D_k)\ :\ C\in\A_C(D_1,\dots,D_k),\ \H^1(C)\le l\right\}.
\ee
\begin{oss}\label{mainoss}
We note that the admissible sets $C$ can be reduced to the set of graphs embedded in $\R^d$. For sake of simplicity, we limit ourselves to the case of three points $D_1,D_2,D_3\in\R^d$ (for the general result see \cite{buruve}). Let $C\in\A_C(D_1,D_2,D_3)$ be such that $\H^1(C)\le l$ and let $\eta:[0,a]\to C$ be a geodesic in $C$ connecting $D_1$ to $D_2$ which we suppose that do not pass through $D_3$. Let $\xi:[0,b]\to C$ be a geodesic in $C$ connecting $D_3$ to $D_1$ and let $l_3\in[0,b]$ be the smallest real number such that $\xi(l_3)\in\eta([0,a])$. We define 
$$\gamma_1=\eta_{|[0,l_1]},\ \gamma_2=\eta(d_C(D_1,D_2)-\cdot)_{|[0,l_2]},\ \gamma_3=\xi_{|[0,l_3]},$$
where $l_1$ and $l_2$ are such that $\eta(l_1)=\xi(l_3)$ and $l_2=d_C(D_1,D_2)-l_1$.

%Then, consider the geodesic $\eta_3$ connecting $D_4$ to $D_1$ and the smallest real number $a_3$ such that $\eta_3(a_3)\in\eta_1([0,a_1])\cup\eta_2([0,a_2])$. Repeating this operation, we obtain a family of geodesics $\eta_i$, $i=1,\dots,k-1$ which intersect each other in a finite number of points. Each of these geodesics can be decomposed in several parts according to the intersection points with the other geodesics (see Figure \ref{f1th1}). 

\begin{figure}[h]
\centerline{\includegraphics[width=6.5cm]{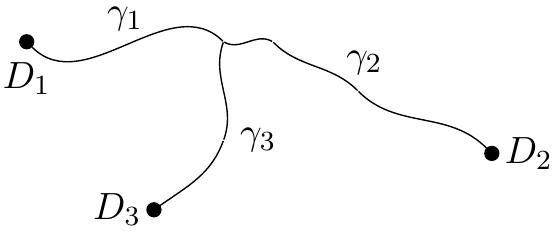}
\includegraphics[width=6.5cm]{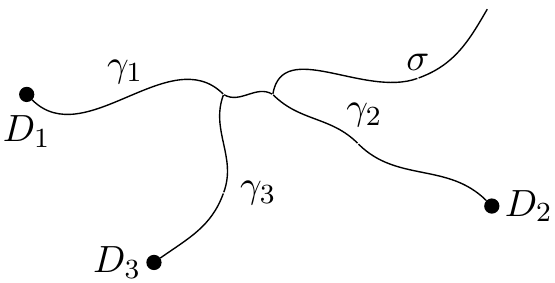}}
\caption{The set $C'$ (on the left) and $\widetilde C$ (on the right).}
\label{f1th1}
\end{figure}

The curves $\gamma_1$, $\gamma_2$ and $\gamma_3$ are geodesics in $C$ which does not intersect each other in internal points (note that it is possible that one of them is degenerate, i.e. constant). Consider the set $C'=\cup_i\,\gamma_i([0,l_i])\subset C$. By construction $C'$ is connected and contains $D_1,D_2$ and $D_3$. Let $w\in H^1_0(C;D_1,D_2,D_3)$ be a positive function and let $v:[0,\H^1(C\setminus C')]\to\R$ be a monotone increasing function such that $|\{v\le\tau\}|=\H^1(\{w\le\tau\}\cap\Gamma)$. By the Polya-Szeg\"o inequality (see \cite[Remark 2.6]{buruve} or \cite{friedlander}), we have 
\be\label{e1t1}
\frac12\int_0^{\H^1(\Gamma)}|v'|^2\,dx-\int_0^{\H^1(\Gamma)}v\,dx
\le\frac12\int_{\Gamma}|w'|^2\,d\H^1-\int_{\Gamma}w\,d\H^1.
\ee

Let $\sigma:[0,\H^1(C\setminus C')]\to\R^d$ be an injective arc-length parametrized curve such that $Im(\sigma)\cap C'=\sigma(0)=x'$, where $x'\in C'$ is the point where $w_{|C'}$ achieves its maximum. Then the closed connected set $\widetilde C=C'\cup\sigma([0,\H^1(C\setminus C')])$ is admissible and has lower energy than $C$. In particular, in problem \eqref{minenergy1} with three fixed points, we can restrict our attention to sets, which are representations of metric graphs (i.e. combinatorial graphs with weighted edges) in $\R^d$. More precisely, we can consider graphs $C$ such that  
\begin{enumerate}
\item $C$ is a tree, i.e. it does not contain any closed loop;
\item $C$ has at most $6$ vertices; if a vertex has degree three or more, we call it Kirchhoff point;
\item there is at most one vertex of degree one for $C$ which is not a Dirichlet point. In this vertex the energy function $w$ satisfies Neumann boundary condition $w'=0$ and so we call it Neumann point.
\end{enumerate}
\end{oss}
In the setting described above, the topology on the set of admissible graphs is quite natural, i.e. we say that $C_n$ converges to $C$, if the weighted connectivity matrices of the graphs $C_n$ converge to that of $C$, where the element $m_{ij}$ of the connectivity matrix $M=(m_{ij})_{ij}$ is equal to the length of the edge connecting the two vertices $V_i$ and $V_j$ with the convention that $m_{ij}=+\infty$ if the there is no edge connecting the two vertices and $m_{ij}=0$, if the two vertices coincide. It is quite clear that with this topology the set of connected metric trees of at most $N$ vertices is compact. On the other hand, as the following example shows, the energy $\E(C,\D)$ is not semi-continuous.
\begin{exam}
Consider the points $D_1=(0,0)$, $D_2=(1,0)$ and $D_3=(2,0)$ and the set $C_n\subset\R^2$ consisting of the graphs of the functions $y(x)=x(x-1)$ for $x\in[0,1]$ and $y_n(x)=-\frac{1}{n}x(x-2)$ for $x\in[0,2]$. Passing to the limit as $n\to\infty$, we have that the arc connecting $D_1$ to $D_3$ passes through the Dirichlet point $D_2$ which causes the energy to suddenly increase.  
\end{exam}
\begin{oss}
The lack of semi-continuity does not necessarily imply the non-existence of a solution of \eqref{minenergy1}, but suggests the nature of a possible counter-example. Following this idea, in \cite{buruve}, was proved that if $\D=\{D_1, D_2, D_3\}\subset\R^2$ is a set of points, with coordinates respectively $(-1,0)$, $(1,0)$ and $(n,0)$, and $l=n+2$ is a given length, then, for $n$ large enough, the problem \eqref{minenergy1} does not have a solution. 
\end{oss}

In order to obtain an existence result for the problem \eqref{minenergy1}, we consider, as in \cite{buruve}, in a larger class of admissible sets.  Indeed, let $\Gamma$ be a combinatorial graph with vertices $\{V_i\}_{i=1,\dots,N}$ and edges $\{e_{ij}\}_{ij}$. We call $\Gamma$ a metric graph, if to each edge $e_{ij}$ is associated a positive real number $l_{ij}$ which we interpret as the length of the edge. Thus, the total length of $\Gamma$ is given by $l(\Gamma):=\sum_{i<j}l_{ij}$.

A function $u:\Gamma\rightarrow\R^n$ on the metric graph $\Gamma$ is a collection of functions $u_{ij}:[0,l_{ij}]\rightarrow\R$, for $1\le i\neq j\le N$, such that:
\begin{enumerate}
\item $u_{ji}(x)=u_{ij}(l_{ij}-x)$, for each $1\le i\ne j\le N$,
\item $u_{ij}(0)=u_{ik}(0)$, for all $\{i,j,k\}\subset\{1,\dots,N\}$.
\end{enumerate}
We say that $u$ is continuous ($u\in C(\Gamma)$), square integrable $u\in L^2(\Gamma)$ or Sobolev $u\in H^1(\Gamma)$, if $u_{ij}$ is respectively continuous, square integrable or Sobolev on each edge $e_{ij}$. We also note that, if $u\in H^1(\Gamma)$, then $|u'|\in L^2(\Gamma)$ and so, we can define
\be\label{engamma}
E(\Gamma;\{V_1,\dots, V_k\})=\min_{u\in H^1_0(\Gamma;\{V_1,\dots, V_k\})}\frac12\int_{\Gamma}|u'|^2\,d\H^1-\int_{\Gamma}u\,d\H^1, 
\ee 
where $H^1_0(\Gamma;\{V_1,\dots, V_k\})$ indicates the subspace of $ H^1(\Gamma)$ of the functions vanishing on each of the vertices $V_1,\dots,V_k$ and we also used the notation
$$\int_{\Gamma}|u'|^2\,d\H^1:=\sum_{ij}\int_0^{l_{ij}}|u_{ij}'|^2\,dx,\qquad\int_{\Gamma}u\,d\H^1:=\sum_{ij}\int_0^{l_{ij}}u_{ij}\,dx.$$

We say that the continuous function $\gamma=(\g_{ij})_{1\le i\ne j\le N}:\Gamma\to\R^d$ is an \emph{immersion} of the metric graph $\Gamma$ into $\R^d$, if for each $1\le i\ne j\le N$ the function $\g_{ij}:[0,l_{ij}]\to\R^d$ is an injective arc-length parametrized curve. Given a set of distinct points $D_1,\dots,D_k\in\R^d$, we define the admissible set $\A(D_1,\dots,D_k)$ as the set of metric graphs $\Gamma$ for which there is an immersion $\gamma:\Gamma\to\R^d$ such that $\gamma(V_i)=D_i$, where $V_1,\dots,V_k$ are vertices of $\Gamma$. In \cite{buruve} the following result was proved.

\begin{teo}\label{graphth}
Consider a set of distinct points $D_1,\dots,D_k\in\R^d$ and a real number $l$ such that there is a closed set $C\subset\R^d$ which contains $D_1,\dots,D_k$ and such that $\HH^1(C)\le l$. Then the following problem has a solution:
\be\label{enim2}
\min\Big\{E(\Gamma;\{V_1,\dots, V_k\})\ :\ \Gamma\in\A(D_1,\dots,D_k),\ l(\Gamma)\le l\Big\}.
\ee
\end{teo}

In some situations, we can use Theorem \ref{graphth} to obtain an existence result for \eqref{minenergy1}.

\begin{prop}
Suppose that $D_1$, $D_2$ and $D_3$ be three distinct, non co-linear points in $\R^d$ and let $l>0$ be a real number such that there exists a closed set of length $l$ connecting $D_1$, $D_2$ and $D_3$. Then the problem \eqref{minenergy1} has a solution.
\end{prop}

\begin{proof}
Let the graph $\Gamma$ be a solution of \eqref{enim2} and let $\gamma:\Gamma\to\R^d$ be an immersion of $\Gamma$ such that $\gamma(V_j)=D_j$ for $j=1,2,3$. Note that if the immersion $\gamma$ is such that the set $\gamma(\Gamma)\subset\R^d$ is represented by the same graph $\Gamma$, then $\gamma(\Gamma)$ is a solution of \eqref{minenergy1} since we have 
$$E(\Gamma;\{V_1,V_2,V_3\})=E(C;D_1,D_2,D_3).$$
Reasoning as in Remark \ref{mainoss}, we can suppose that $\Gamma$ is obtained by a tree $\Gamma'$ with vertices $V_1$, $V_2$ and $V_3$ by attaching a new edge (with a new vertex in one of the extrema) to some vertex or edge of $\Gamma'$. Since we are free to choose the immersion of the new edge, we only need to show that we can choose $\gamma$ in order to have that the set $\gamma(\Gamma')$ is represented by $\Gamma'$. On the other hand we have only two possibilities for $\Gamma'$ and both of them can be seen as embedded graphs in $\R^d$ with vertices $D_1,D_2$ and $D_3$.
\end{proof}

\begin{oss}
Similarly to the existence proof of a classical optimal graph of Proposition above we believe that a more general result should hold: if $D_1,\dots, D_k$ are $k$ distinct points in $\R^d$ such that none of them can be expressed as a convex combination of the others, then \eqref{minenergy1} has a solution. We do not yet have a complete proof of this fact.
\end{oss}

\begin{exam}\label{example2}
Let $D_1$ and $D_2$ be two distinct points in $\R^d$ and let $l\ge|D_1-D_2|$ be a real number. Then the optimization problem \eqref{enim2} has a solution $\Gamma$ which is actually a classical graph $C$ given by the connected set (see Figure \ref{f1ex2})
$$C=[D_1,D_2]\cup\left[\frac{D_1+D_2}{2},D_3\right]\qquad\hbox{with }\left|D_3-\frac{D_1+D_2}{2}\right|=l-|D_1-D_2|.$$

\begin{figure}[h]
\centerline{\includegraphics[width=14.0cm]{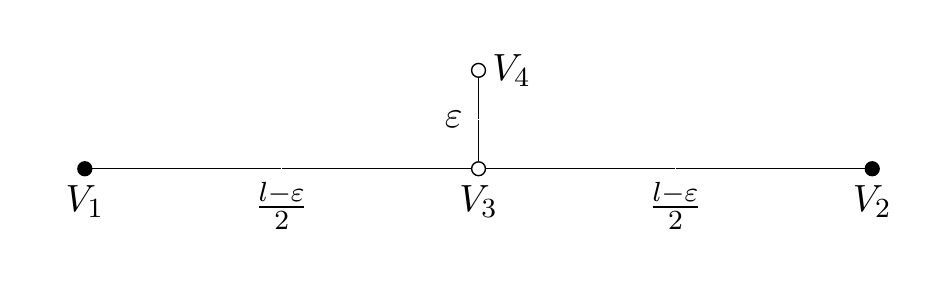}}
\caption{The optimal graph with two Dirichlet points.}
\label{f1ex2}
\end{figure}
\end{exam}

\begin{exam}\label{example3}
Let $D_1$, $D_2$ and $D_3$ be the vertices of an equilateral triangle of side $1$ in $\R^2$, i.e.
$$D_1=\left(-\frac{\sqrt3}{3},0\right),\ D_2=\left(\frac{\sqrt3}{6},-\frac12\right),\ D_3=\left(\frac{\sqrt3}{6},\frac12\right).$$
We study in \cite{buruve} the problem \eqref{minenergy1} with $\D=\{D_1,D_2,D_3\}$ and $l>\sqrt 3$. We show that the solutions may have different qualitative properties for different $l$ and that there is always a symmetry breaking phenomenon, i.e. the solutions do not have the same symmetries as the initial configuration $\D$. Indeed, an explicit estimate of the energy shows that (see Figure \ref{f2ex3}):

\begin{enumerate}
\item if $\sqrt3\le l\le1+\sqrt3/2$, we have that the solution of the problem \eqref{minenergy1} with $\D=\{D_1,D_2,D_3\}$ is of the form $\Gamma_1$;
\item if $l>1+\sqrt3/2$, then the solution of the problem \eqref{enim2} with $\D=\{D_1,D_2,D_3\}$ is of the form $\Gamma_3$.
\end{enumerate}
\end{exam}

\begin{figure}[h]
\centerline{\includegraphics[width=9.0cm]{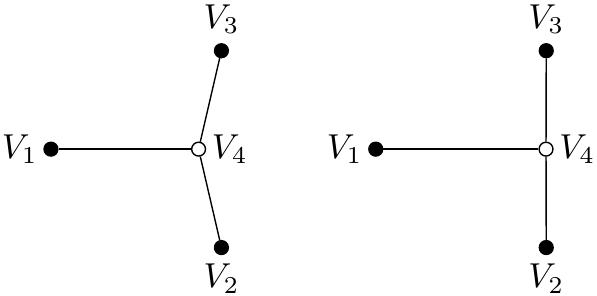}}
\centerline{\includegraphics[width=9.0cm]{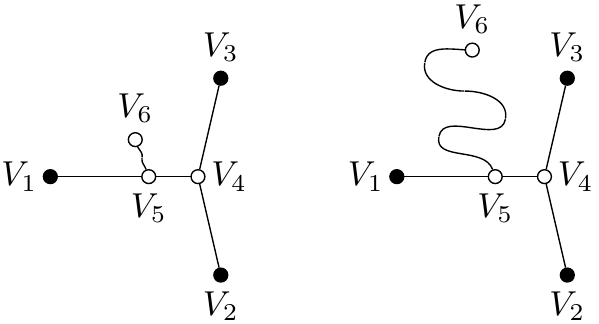}}
\caption{The optimal graphs for $l<1+\sqrt{3}/2$, $l=1+\sqrt3/2$, $l>1+\sqrt3/2$ and $l>>1+\sqrt3/2$.}
\label{f2ex3}
\end{figure}

%%%%%%%%%%%%%%%%%%%%%%%%%%%%%%
\section{Spectral optimization for Schr\"odinger operators}\label{s4}

Consider a bounded open set $\O\subset\R^d$ and a function $f\in L^2(\O)$. The Dirichlet energy related to a potential $V\ge0$ on $\O$ is defined as
\be\label{enV}
\E_f(V)=\min_{u\in H^1_0(\O)}\frac12\int_\O|\nabla u|^2\,dx+\frac12\int_\O u^2V\,dx-\int_\O uf\,dx.
\ee
A natural question, analogous to the problems considered in Section \ref{s2}, is the optimization of $\E_f(V)$ under some integral constraint on $V$, of the form $\int V^p\,dx\le1$. It is clear, from the definition of $\E_f$, that for $p>0$ the minimum is achieved by $V=0$. On the contrary, maximizing the energy under the same constraints gives the following results.

\begin{itemize}
\item If $p<1$ a maximizing potential does not exist. In fact, for any $p<1$, one may construct a sequence of functionals $V_n$ such that $\int V_n^p\,dx=1$ and $\E_f(V_n)\to 0$, as $n\to\infty$.

\item If $p>1$ the optimal potential $V_p$ exists and is given by
$$V_p=|u|^{2/(p-1)}\cdot\Big(\int_\O|u|^{2p/(p-1)}\,dx\Big)^{-1/p}$$
where $u$ is the solution of the minimum problem
$$\min_{u\in H^1_0(\O)}\frac12 \int_\O |\nabla u|^2\,dx+\frac12\left(\int_\O|u|^{2p/(p-1)}\,dx\right)^{(p-1)/p}-\int_\O uf\,dx,$$
which is also the strong solution of $-\Delta u+uV_p=f$ in $\O$.

\item If $p=1$ the optimal potential $V_1$ exists and is given by
$$V_1=\frac fM\left(1_{\omega_+}-1_{\omega_-}\right),$$
where $M=\|u_1\|_{L^\infty(\O)}$, $\omega_+=\{u_1=M\}$, $\omega_-=\{u_1=-M\}$, and $u_1\in H^1_0(\O)\cap H^2(\O)$ is the unique minimizer of the functional $J_1:L^2(\O)\to\R$, defined as
$$J_1(u):=\frac12 \int_\O |\nabla u|^2\,dx+\frac12\|u\|_{L^\infty(\O)}^2-\int_\O uf\,dx.$$
In particular, we have
$$\int_{\omega_+}f\,dx-\int_{\omega_-}f\,dx=M,\qquad f\ge0\hbox{ on }\omega_+,\qquad f\le0\hbox{ on }\omega_-\;.$$
\end{itemize}
 
\begin{exam}
Let $\O=(-1,1)$ and $f$ be a positive constant on $\O$. Then $u_1$ is positive and, by a symmetrization argument, it is also radially symmetric and decreasing. Thus, $\omega_+=(-a,a)$ for some $a\in (0,1)$ and since $|\omega_+|M=1$, we have that $a=\frac{1}{2M}$. Since $u'(\frac{1}{2M})=0$ and $u''=-f$ on $(\frac{1}{2M},1)$, we have that $(1-\frac{1}{2M})^2 f=2M$, which uniquely determines $M$ and so, the optimal potential $V_1=\frac{1}{M}1_{(-\frac{1}{2M},\frac{1}{2M})}$.
\end{exam}

When $p<0$ the minimization problem
\be\label{minpb}
\min\left\{\E_f(V)\ :\ V:\O\to[0,+\infty],\ \int_\O V^p\,dx=1\right\},
\ee
becomes meaningful.

\begin{prop}\label{min1}
Let $\O\subset\R^d$ be a bounded open set and let $f\in L^2(\O)$. Then, for every $p<0$, the problem \eqref{minpb} has a solution.
\end{prop}

\begin{proof}
By the definition of $\E_f(V)$, interchanging the two min operators, we find that the optimal potential $V_p$ is given by
\be\label{pneg}
V_p=|u|^{2/(p-1)}\cdot\Big(\int_\O|u|^{2p/(p-1)}\,dx\Big)^{-1/p}
\ee
where $u$ is the solution of the minimum problem
\be\label{eqpneg}
\min_{u\in H^1_0(\O)}\frac12 \int_\O |\nabla u|^2\,dx+\frac12\left(\int_\O|u|^{2p/(p-1)}\,dx\right)^{(p-1)/p}-\int_\O uf\,dx.
\ee
Note that, since $p<0$, the quantity $q=2p/(p-1)$ is such that $0<q<2$. The existence of a solution for problem \eqref{eqpneg} is straightforward, which gives the existence of the optimal potential $V_p$ through equality \eqref{pneg}.
\end{proof}

When we consider more general cost functionals $F(V)$, like for instance spectral costs depending on the eigenvalues of the Schr\"odinger operator $-\Delta+V$, the proof above cannot be repeated; nevertheless, using finer tools like $\gamma$-convergence for Dirichlet problems, the following more general result can be obtained (see \cite{bugeruve}).

\begin{teo}
Consider a cost functional $F:\B_+(\O)\to\R$, where $\B_+(\O)$ denotes the space of Borel measurable positive functions on $\O$. Suppose that $F$ is
\begin{enumerate}
\item increasing, i.e. $F(V)\ge F(W)$, whenever $V\ge W$;
\item lower semi-continuous with respect to strong convergence of the resolvents 
$$R_V=(-\Delta +V)^{-1}:L^2(\O)\to L^2(\O).$$
\end{enumerate} 
Then, for any $p<0$, the optimization problem
$$\min\left\{F(V)\ :\ V:\O\to[0,+\infty],\ \int_\O V^p\,dx=1\right\},$$
has a solution.
\end{teo}

\bigskip

\begin{ack}
The second author wish to thank Gian Maria Dall'Ara for the useful discussions.
\end{ack}

%%%%%%%%%%%%%%%%%%%%%%%%%%%%%%%%%%%%%%%%%%%%%%%%%%

\bigskip
{\small
\begin{minipage}[t]{6.9cm}
Giuseppe Buttazzo\\
Dipartimento di Matematica\\
Universit\`a di Pisa\\
Largo B. Pontecorvo, 5\\
56127 Pisa - ITALY\\
{\tt buttazzo@dm.unipi.it}
\end{minipage}
\begin{minipage}[t]{6.9cm}
Bozhidar Velichkov\\
Scuola Normale Superiore di Pisa\\
Piazza dei Cavalieri, 7\\
56126 Pisa - ITALY\\
{\tt b.velichkov@sns.it}
\end{minipage}}

\end{document}